\documentclass[12pt]{amsart}
\usepackage[utf8]{inputenc}
\usepackage{amssymb}
\usepackage{amsmath}
\usepackage{amsfonts}
\usepackage{amsthm}
\usepackage{tikz}
\usepackage{algorithmic}
\usepackage{asymptote}
\usepackage{url}
\usepackage{graphicx}
\usepackage{amssymb, amsmath, amsthm}
\usepackage{color,soul}
\usepackage{polynom}
\usepackage{pst-node}
\usepackage{ytableau}
\usepackage{mathdots}
\usepackage{subcaption}
\usepackage{comment}
\usepackage{xcolor}
\usepackage{float}
\usepackage{setspace}
\usepackage{caption}
\newtheorem{theorem}{Theorem}[section]

\newtheorem{proposition}[theorem]{Proposition}

\usepackage{times}

\theoremstyle{definition}

\newtheorem{defn}[theorem]{Definition}

\newtheorem{innercustomthm}{Theorem}
\newenvironment{customthm}[1]
  {\renewcommand\theinnercustomthm{#1}\innercustomthm}
  {\endinnercustomthm}

\newtheorem{innercustomcor}{Corollary}
\newenvironment{customcor}[1]
  {\renewcommand\theinnercustomcor{#1}\innercustomcor}
  {\endinnercustomcor}
  
\newtheorem{innercustomprop}{Proposition}
\newenvironment{customprop}[1]
  {\renewcommand\theinnercustomprop{#1}\innercustomprop}
  {\endinnercustomprop}
  
\newtheorem{innercustomlem}{Lemma}
\newenvironment{customlem}[1]
  {\renewcommand\theinnercustomlem{#1}\innercustomlem}
  {\endinnercustomlem}

\usepackage{fancyhdr}

\begin{document}

\title{Complementary Numerical Sets}

\author{Matthew Guhl}
\author{Jazmine Juarez}
\author{Vadim Ponomarenko}
\author{Rebecca Rechkin}
\author{Deepesh Singhal}


\begin{abstract}
A numerical set $S$ is a cofinite subset of $\mathbb{N}$ which contains $0$.  We use the natural bijection between numerical sets and Young diagrams to define a numerical set $\widetilde{S}$, such that their Young diagrams are complements.  We determine various properties of $\widetilde{S}$, particularly with an eye to closure under addition (for both $S$ and $\widetilde{S}$), which promotes a numerical set to become a numerical semigroup.
\end{abstract}

\maketitle

\section{Introduction}
We denote by $\mathbb{N}$ the set of naturals, i.e. the set of non-negative integers.  A \emph{numerical set} is a cofinite subset of $\mathbb{N}$ that contains 0. The natural numbers that are missing from a numerical set $S$ are called its gaps, the collection of all gaps is denoted by $Gap(S)$. The largest gap is called the \emph{Frobenius number} and is denoted by $F(S)$. The number of gaps is called the \emph{genus} and is denoted by $g(S)$. The smallest non-zero element is called its \emph{multiplicity} and is denoted by $m(S)$. A \emph{numerical semigroup} is a numerical set which is closed under addition. Given a numerical set there is a natural way of constructing a numerical semigroup from it, which is called the \emph{atomic monoid} of the set or the \emph{associated semigroup} of the set. We denote the associated semigroup of a numerical set $S$ as $A(S)$, and it is given by
$$A(S)=\{s\in S \mid s+S \subseteq S\}.$$
It is straightforward to show that $A(S)$ is in fact a numerical semigroup and it has same Frobenius number as $S$. Also note that if $S$ was a numerical semigroup, then the associated semigroup of $S$ is itself, since $S$ is closed under addition. This operation of associated semigroup was defined in \cite{AS1} and has been studied in several recent papers \cite{Kaplan et al, AS2 MM, AS3 Grobenius, AS4, Deeprilium 1, Deeprilium 2}. See \cite{NS Book} for a general reference on numerical semigroups.

The \emph{atoms} of a numerical semigroup $S$ are the positive elements of $S$ that cannot be written as a sum of two positive elements of $S$. The number of atoms of $S$ is called its \emph{embedding dimension}, is denoted by $e(S)$. It is known that $e(S)\leq m(S)$, and $S$ is said to be of \emph{max embedding dimension} if $e(S)=m(S)$. An atom (resp. positive element) of $S$ is called a small atom (resp. small element) of $S$ if it is smaller than the Frobenius number of $S$.
Note that if $S$ is a numerical set for which $A(S)$ has no small atoms then $A(S)=\{0,F(S)+1\rightarrow\}$. Here the $\rightarrow$ indicates that all numbers after $F(S)+1$ are in the numerical set. \cite{AS2 MM} computes the density of such numerical sets among all numerical sets of a given Frobenius number. There are $2^{f-1}$ numerical sets with Frobenius number $f$, they prove that the following limit exists.
$$\lim_{f\to\infty}\frac{\#\{S\mid A(S)=\{0,f+1\rightarrow\}\}}{2^{f-1}}=\gamma.$$
The limit $\gamma$ is approximately $0.48$ meaning, for around $48\%$ of the numerical sets $A(S)$ has no small atoms.
The authors of \cite{AS3 Grobenius} study the numerical sets for which $A(S)$ has one small element and they conjecture that for a fixed $l$ the following limit exists
$$\lim_{f\to\infty}\frac{\#\{S\mid A(S)=\{0,f-l,f+1\rightarrow\}\}}{2^{f-1}}=\gamma_l.$$
In \cite{AS4} this conjecture is proved and extended to the case of $n$ small elements.
If $A(S)$ has one small atom then $A(S)$ must be of the form $m\mathbb{N}\cup \{f+1\rightarrow\}$. With $m$ fixed, \cite{Deeprilium 2} enumerates the number of such numerical sets and shows that it is a quasi-polynomial in $f$. 

Kaplan et al. in \cite{Kaplan et al} showed that numerical sets have a bijective correspondence to Young diagrams. A \emph{Young diagram} is an array given by stacking rows of squares of varying length, but with the property that the rows are non-increasing in length as they progress down. An example is given in Figure \ref{fig:YDiagram}.
We now describe how to get a Young diagram from a given numerical set. A Young diagram is determined by the path connecting its bottom left point to its top right point.
In order to get this path we start at the origin and consider the natural numbers starting from $0$. We take a step left for every natural number that is in the numerical set and take a step upwards for every natural number not in the numerical set. We stop once we reach the Frobenius number. The path thus obtained will enclose a Young diagram. For example the given Young diagram in Figure \ref{fig:YDiagram} would be obtained if our numerical set is $\{0,2,4,7,8,10,12\rightarrow\}$.
This process is reversible and it gives a one to one correspondence between numerical sets and Young diagrams.

\begin{figure}
    \centering
    \includegraphics[width=0.5\textwidth]{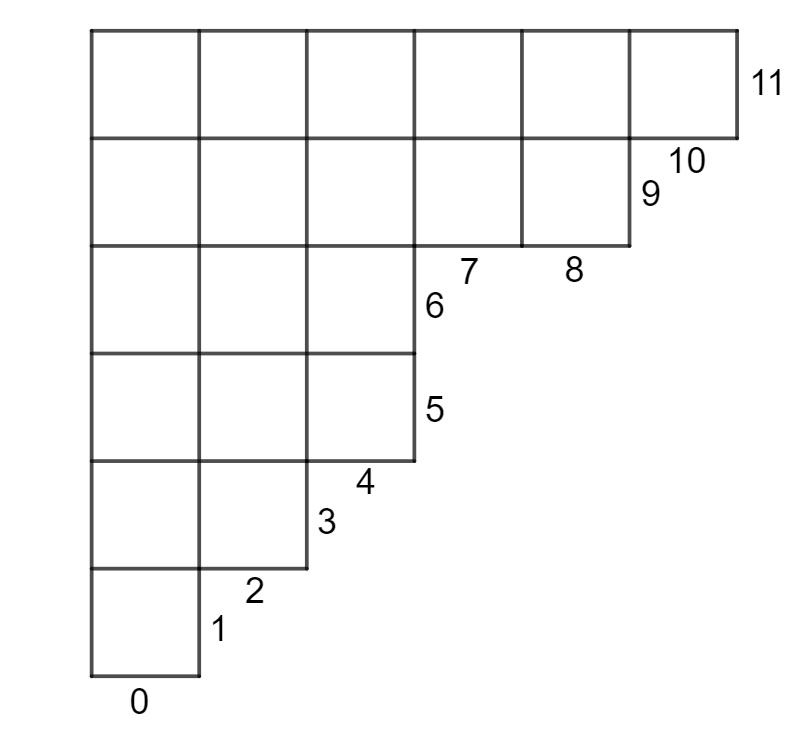}
    \caption{Young diagram corresponding to $\{0,2,4,7,8,10,12\rightarrow\}$}
    \label{fig:YDiagram}
\end{figure}

The \emph{complement} of a Young diagram is found by completing the rectangular grid with length and width of the first row and first column respectively and then rotating the other piece by $180^{\circ}$. Look at Figure \ref{fig complement} for an example.
We can start with a numerical set $S$, consider its Young diagram, take its complement and consider the numerical set associated with the complement. The numerical set thus obtained is called the complement of $S$ and is denoted by $\widetilde{S}$.
In our example with $S=\{0,2,4,7,8,10,12\rightarrow\}$ we have $\widetilde{S}=\{0,2,3,6,8,10\rightarrow\}$.
It should be noted that applying the complement operation twice does not lead back to the original numerical set and several numerical sets can have the same complement.
In this paper, we continue the study of numerical semigroups with no small atoms, with one small atom, and with one small element, using the Young diagram tools described above.  Our main results are as follows.

\begin{figure}
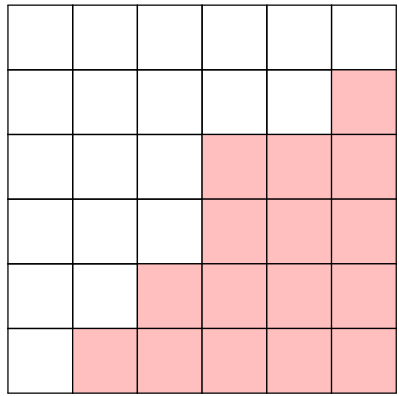

\begin{center}
\ytableausetup{mathmode, boxsize = 2em}
\begin{ytableau}
    \phantom{a} & \phantom{a} & \phantom{a} & \phantom{a} & \phantom{a} &  \\
    \phantom{a} & \phantom{a} & \phantom{a} &  & &*(pink) \\
    \phantom{a} & \phantom{a} &&*(pink)&*(pink)&*(pink) \\
    \phantom{a} & \phantom{a} & &*(pink)&*(pink)&*(pink) \\
    \phantom{a} &  &*(pink)&*(pink)&*(pink)&*(pink) \\
    \phantom{a}&*(pink)&*(pink)&*(pink)&*(pink)&*(pink)
    
\end{ytableau}
\end{center}
\captionsetup{justification=centering,margin=2cm}
\caption{\footnotesize The complement of the Young diagram is in \textcolor{pink}{pink}. Notice it is also a Young diagram when rotated by 180 degrees}
\label{fig complement}
\end{figure}

\begin{customthm}{\ref{S NS}}
Given a numerical semigroup $S$, $A(\widetilde{S})$ has at most one small atom. Moreover if $S$ has more than one small atom then $A(\widetilde{S})$ has no small atoms.
\end{customthm}
\begin{customthm}{\ref{Stilde NS}}
Let $S$ be a numerical set for which $\widetilde{S}$ is a numerical semigroup. Then $A(S)$ has at most one small atom. Moreover if $S$ is not a numerical semigroup then $A(S)$ has at most one small element.
\end{customthm}

\begin{customcor}{\ref{both NS}}
Given a numerical semigroup $S$, its complement $\widetilde{S}$ is a numerical semigroup if and only if $S$ has at most one small atom.
\end{customcor}

\section{Structure of Complementary Numerical Sets}

Young diagrams of numerical sets are closely related to their associated semigroups. Every box on the Young diagram has a \emph{hook}, which consists of all the boxes below and to the right of the given box and that box itself. The number of boxes in the hook is called the \emph{hook length} of that box. This is seen in Figure \ref{Hook numbers}, where the number in each square is the hook length of that square, and a particular hook with hook length 5 is emphasized in green. Kaplan et al. \cite{Kaplan et al} have shown that the hook lengths of a numerical set's Young diagram correspond precisely to the gaps of its associated semigroup.

\begin{figure}[h]
\begin{center}
\ytableausetup{mathmode, boxsize = 2em}
\begin{ytableau}
    11 & 9 & 7 & 4 & 3 & 1 \\
    9 & 7 & *(green)5 & *(green)2 & *(green)1 \\
    6 & 4 & *(green)2 \\
    5 & 3 & *(green)1 \\
    3 & 1  \\
    1
\end{ytableau}
\end{center}
\captionsetup{justification=centering,margin=2cm}
\caption{\footnotesize Since there are 2 squares to the right and 2 below, the hook length at this position is 5.}
\label{Hook numbers}
\end{figure}

There are often many numerical sets that have the same associated numerical semigroup which means that the set of hook numbers does not uniquely determine a Young diagram.
Herman and Chung \cite{Herman and Chung} found that the hook multi-set is also not unique to a particular Young diagram. They did however, find that a numerical set is fully characterized by its hook multi-set accompanied by the hook multi-set of the complement of the Young diagram.
We will now give a more direct description of the complement numerical set.

\begin{defn}
Let S be a numerical set. We denote its Base $B(S)$ as its biggest element smaller than the Frobenius number namely
$$B(S)=\max\{s \in S \mid s<F(S)\}.$$
\end{defn}

\begin{customthm}{2.2}\label{SetDefofComplement}
Let $S$ be a numerical set with complement $\widetilde{S}$. Then,
$$\Tilde{S}=\{B(S)-s \mid s\in S \textnormal{ and } s \leq B(S) \}\cup\{B(S)\rightarrow\}.$$
\end{customthm}
\begin{proof}
Consider the following Young diagram:
\begin{center}
\ytableausetup{mathmode, boxsize = 2.5em}
\begin{ytableau}
    \phantom{a} & \phantom{a} & \phantom{a} & \phantom{a} & \phantom{a} & B \\
    \phantom{a} & \phantom{a} & \phantom{a} & B-l & \dots & *(pink) 0 \\
    \phantom{a} & \phantom{a} & \vdots & *(pink)l & *(pink) \phantom{a} & *(pink)\phantom{a} \\
    \phantom{a} & \phantom{a} & B-p & *(pink)\phantom{a} & *(pink)\phantom{a} & *(pink)\phantom{a} \\
    \phantom{a} & B-n & *(pink) p & *(pink)\phantom{a} & *(pink)\phantom{a} & *(pink)\phantom{a} \\
    \vdots & *(pink) n & *(pink)\phantom{a} & *(pink)\phantom{a} & *(pink)\phantom{a} & *(pink)\phantom{a} 
\end{ytableau}
\end{center}
The $0$ of $\widetilde{S}$ corresponds to $B(S)$ of $S$ in the Young diagram. Following the path from there we see that the $l$ of $\widetilde{S}$ corresponds to $B(S)-l$ of $S$ in the diagram as long as $0\leq l\leq B(S)$. Since horizontal line segments remain horizontal upon a $180^{\circ}$ rotation we see that for $0\leq l\leq B(S)$, $l\in \widetilde{S}$ if and only if $B(S)-l\in S$.
The $0$ of $S$ corresponds to $B(S)$ of $\widetilde{S}$, the Young diagram of $\widetilde{S}$ finishes before this. This means that $[B(S),\infty)\subseteq \widetilde{S}$.

\end{proof}

\noindent We now present some immediate properties of $\widetilde{S}$.
\begin{customprop}{2.3}
Let $S$ be a numerical set with complement $\widetilde{S}$. Then:\\
(a) $F(\widetilde{S}) \le B(S) - 1 \le F(S) - 2$; further, if $1\notin S$ then $F(\widetilde{S})=B(S)-1$.\\
(b) $g(\widetilde{S})=g(S)+B(S)-F(S).$\\
(c) $B(\widetilde{S})\leq B(S)-m(S)$, moreover equality holds if and only if $1\not\in S$.
\end{customprop}
\begin{proof}
Parts (a) and (c) directly follow from Theorem \ref{SetDefofComplement}. For part (b) Theorem \ref{SetDefofComplement} implies that $|Gap(S)\cap [0,B(S)]|=g(\widetilde{S})$, moreover $Gap(S)\cap [B(S)+1,\infty)=[B(S)+1,F(S)]$. Therefore $g(S)=g(\widetilde{S})+(F(S)-B(S))$.
\end{proof}
\noindent
We now characterize which numerical sets arise as complements of numerical semigroups.
\begin{customprop}{2.4}
Let $T$ be a numerical set. There exists a numerical semigroup $S$ for which $T=\widetilde{S}$ if and only if\\
(a) $F(T)\not\in T+T$; and\\
(b) for every $x,y\in T\cap [0,F(T)]$ with $x+y> F(T)$, we have $x+y-F(T)-1\in T$.
\end{customprop}
\begin{proof}
First suppose we have a numerical semigroup $S$ for which $T=\widetilde{S}$. If there are $x,y\in T$ for which $x+y=F(T)$ then $x,y\leq F(T)=B(S)-1$. This means that $B(S)-x, B(S)-y\in S$, and since $S$ is a numerical semigroup we have $2B(S)-x-y\in S$. But
$$2B(S)-x-y=2B(S)-(B(S)-1)=B(S)+1.$$
However $B(S)<B(S)+1\leq F(S)$, so $B(S)+1$ cannot be in $S$ and we have a contradiction. Next suppose $x,y\in T\cap [0,F(T)]$ with $x+y>F(T)$. We have $B(S)-x, B(S)-y\in S$ and hence $B(S)-(x+y-B(S))\in S$. And since $0\leq x+y-B(S)\leq 2F(T)-B(S)=B(S)-2$ we get that $x+y-B(S)\in \widetilde{S}=T$ and of course $x+y-B(S)=x+y-F(S)-1$.

Conversely if $T$ is a numerical set that satisfies the two conditions then define
$$S=\{F(T)+1-x\mid x\in T\cap [0,F(T)+1]\}\cup \{F(T)+3\rightarrow\}.$$
Clearly $F(S)=F(T)+2$, $B(S)=F(T)+1$ and therefore $T=\widetilde{S}$. We need to show that $S$ is a numerical semigroup. Given positive $a,b\in S$, if $a+b\geq F(T)+3$ then clearly $a+b\in S$. Otherwise $a+b\leq F(T)+2$, so $a,b\leq F(T)+1$ and hence $F(T)+1-a,F(T)+1-b\in T$. Now
$$2F(T)+2-a-b=F(T)+1-a+F(T)+1-b\neq F(T).$$
Which means that $a+b\neq F(T)+2$ and hence $a+b\leq F(T)+1$. Let $x=F(T)+1-a$, $y=F(T)+1-b$ so $x,y\in T\cap [0,F(T)]$ and $x+y=2F(T)+2-a-b\geq F(T)+1$ and hence $x+y-F(T)-1\in T$. Finally $x+y-F(T)-1\leq F(T)+1$ and hence $F(T)+1-(x+y-F(T)-1)\in S$. Since $F(T)+1-(x+y-F(T)-1)=a+b$, we see that $S$ is indeed a numerical semigroup and we are done.
\end{proof}

\section{$A(\Tilde{S})$ when $S$ is a numerical semigroup}
This section will prove Theorem \ref{S NS}. Toward this end, we consider a numerical semigroup $S$ and study the associated semigroup of its complement. Note that if $S$ has a single small atom then it must be the multiplicity $m(S)$. Since $S$ is closed under addition it must also have all multiples of $m(S)$. However if any of the elements in $S\cap [0,F(S)]$ is not a multiple of $m(S)$, then the smallest such element would be another small atom which is impossible. Thus if $S$ is a numerical semigroup with a single small atom then $S\cap [0,F(S)]$ consists precisely of all the multiples of $m(S)$ that are in that range. On the other hand if $S$ has more than one small atom then $\cap [0,F(S)]$ must contain elements that are not multiples of $m(S)$.

\begin{theorem}\label{S NS}
Given a numerical semigroup $S$, $A(\widetilde{S})$ has at most one small atom. Moreover if $S$ has more than one small atom then $A(\widetilde{S})$ has no small atoms.
\end{theorem}
\begin{proof}
Set $m=m(S)$ for convenience. If $S$ has no small atoms then $\widetilde{S}=\mathbb{N}$ is a numerical semigroup and has no small atoms. Next if $S$ has exactly one small atom then
$$S\cap [0,B(S)]=\{lm\mid 0\leq l\leq N\}.$$
Where $B(S)=Nm$. Theorem \ref{SetDefofComplement} implies that $\widetilde{S}=m\mathbb{N}\cup \{Nm\rightarrow\}$, which is a numerical semigroup and has at most one small atom. Now for the rest of the proof assume that $S$ has at least two small atoms.

Since $S$ is a numerical semigroup we have $1\not\in S$, so by Theorem \ref{SetDefofComplement} we see that $F(\widetilde{S})=B(S)-1$, also $B(\widetilde{S})=B(S)-m$.
We will first show that all elements in $A(\widetilde{S})\cap [0,F(\widetilde{S})]$ must be multiples of $m$.
Consider some $x\in \widetilde{S}$ with $1\leq x< B(S)-1$ and $m\nmid x$. Say $x=km+r$ with $1\leq r\leq m-1$. Now since $(k+1)m\in S$ and $(k+1)m<B(S)$ (remember that $x\leq B(\widetilde{S})=B(S)-m$) we have $B(S)-(k+1)m\in\widetilde{S}$ by Theorem \ref{SetDefofComplement}. Now
$$x+(B(S)-(k+1)m)=B(S)-(m-r).$$
Since $r\neq 0$ we know that $m-r\not\in S$ and hence $B(S)-(m-r)\not\in \widetilde{S}$. This shows that $x\not\in A(\widetilde{S})$.

Now consider some $x\in\widetilde{S}$ of the form $x=l m$ for some $l\geq 1$ such that $x<B(S)$. Let $n m$ be the largest multiple of $m$ in $\Tilde{S}$ such that $nm<B(S)$. We have two cases.

Case 1: Suppose $(n+1)m \notin \Tilde{S}$. By Theorem \ref{SetDefofComplement}, $B(S)-nm\in S$. Then since $S$ is closed under addition we have
$$B(S)-(n+1-l)m=B(S)-nm+(l-1)m\in S.$$
This in turn implies that $(n+1-l)m\in \widetilde{S}$. And since
$$x+(n+1-l)m=(n+1)m\not\in\widetilde{S}$$
we see that $x\not\in A(\widetilde{S})$.

Case 2: Now suppose $(n+1)m \in \widetilde{S}$. The maximality of $n$ implies that $(n+1)m\geq B(S)$. This means that $(n+1)m\geq B(S)>nm$.
Now since $nm\in\widetilde{S}$ we have $B(S)-nm\in S$, but $0<B(S)-nm\leq m$. Therefore we must have $B(S)-nm=m$ i.e. $B(S)=(n+1)m$.

Since $S$ has at least two small atoms, there exists a $y\in S$ such that $m\nmid y$ and $y<F(S)$. Choose the smallest such $y$. Let $z=B(S)-y$ so $z$ is the largest number in $\widetilde{S}\cap [0,B(S)]$ that is not a multiple of $m$. Say $z=jm+r'$ with $0<r'<m$.

If $l\geq n-j$, then $y+(l-n+j)m\in S$ and hence
$$(n-l)m+r'=z-(l-n+j)m=B(S)-(y+(l-n+j)m)\in\widetilde{S}.$$
Now $x+(n-l)m+r'=nm+r'$. We know that $B(S)-(nm+r')=m-r'\not\in S$ and hence $nm+r'\not\in \widetilde{S}$. This means that $x\not\in A(\widetilde{S})$.

Now consider the case when $l<n-j$. In this case
$$x+z=lm+jm+r'\leq (n-1)m+r'<nm<B(S).$$
This means that $x+z<B(S)$ and $x+z$ is not divisible by $m$. This means that $x+z\not\in \widetilde{S}$ and hence $x\not\in A(\widetilde{S})$.
We conclude that $x\not\in A(\widetilde{S})$ in all cases and hence $A(\widetilde{S})$ has no small atoms provided $S$ has at least $2$ small atoms.
\end{proof}

During the proof we showed that if $S$ has at most one small atom then $\widetilde{S}$ is a numerical semigroup. This proves one direction of Corollary \ref{both NS}.

\section{$A(S)$ when $\widetilde{S}$ is a numerical semigroup}
In this section we will prove Theorem \ref{Stilde NS}. We consider numerical sets $S$ for which $\widetilde{S}$ is a numerical semigroup and study $A(S)$. We have already seen one such scenario in which $S$ is a numerical semigroup with at most one small atom.

\begin{theorem}\label{Stilde NS}
Let $S$ be a numerical set for which $\widetilde{S}$ is a numerical semigroup. Then $A(S)$ has at most one small atom. Moreover if $S$ is not a numerical semigroup then $A(S)$ has at most one small element.
\end{theorem}
\begin{proof}
Set $\widetilde{m}=m(\widetilde{S})$ for convenience. Let $b=Max(S\cap [0,B(S)-1])$. By Theorem \ref{SetDefofComplement} we know that $\widetilde{m}=B(S)-b$.
We will first show that all elements of $A(S)\cap [0,B(S)-1]$ must be multiples of $\widetilde{m}$. Consider some $x\in S\cap [0,B(S)-1]$ that is not a multiple of $\widetilde{m}$. Say $x=q\widetilde{m}+r$ with $1\leq r\leq \widetilde{m}-1$. We know that $(q+1)\widetilde{m}\in\widetilde{S}$ and $q\widetilde{m}\leq x\leq b=B(S)-\widetilde{m}$ which implies that $B(S)-(q+1)\widetilde{m}\in S$. Now
$$B(S)-(\widetilde{m}-r)=x+(B(S)-(q+1)\widetilde{m}).$$
However $b=B(S)-\widetilde{m}<B(S)-(\widetilde{m}-r)<B(S)$ so $B(S)-(\widetilde{m}-r)\not\in S$ which implies that $x\not\in A(S)$.

Let $n\widetilde{m}$ be the largest multiple of $\widetilde{m}$ in $S\cap [0,B(S)-1]$. There are two cases.\\
Case 1: Consider the case when $(n+1)m(S)\not\in S$. Suppose $x=q\widetilde{m}\in S$ with $0<x\leq b$. We know that $B(S)-n\widetilde{m}\in\widetilde{S}$ so $B(S)-(n+1-q)\widetilde{m}$ is also in $\widetilde{S}$. This implies that $(n+1-q)\widetilde{m}$ is in $S$. We see that $x+(n+1-q)\widetilde{m}=(n+1)\widetilde{m}$ is not in $S$ and hence $x\notin A(S)$. This implies that $A(S)\cap [1,B(S)-1]=\emptyset$.\\
Case 2: The other case is when $(n+1)\widetilde{m}\in S$. The maximality of $n$ implies that $(n+1)\widetilde{m}\geq B(S)$. This means that $B(S)-n\widetilde{m}\leq \widetilde{m}$, but we know that $B(S)-n\widetilde{m}$ is in $\widetilde{S}$ and it is positive. Therefore $B(S)-n\widetilde{m}=\widetilde{m}$ i.e. $(n+1)\widetilde{m}=B(S)$. For $0\leq j\leq n+1$, $(n+1-j)\widetilde{m}\in\widetilde{S}$ and hence $j\widetilde{m}=B(S)-(n+1-j)\widetilde{m}\in S$.

Now consider the case when all elements in $S\cap [0,B(S)]$ are multiples of $\widetilde{m}$. This implies that $m(S)=\widetilde{m}$. Now if $F(S)<(n+2)\widetilde{m}=B(S)+\widetilde{m}$ then it follows that $S$ is a numerical semigroup and it has at most one small atom. On the other hand if $B(S)+\widetilde{m}\leq F(S)$, then for each $q$ with $1\leq q\leq n+1$ we have $q\widetilde{m}+(n+2-q)\widetilde{m}=(n+2)\widetilde{m}$ which is not in $S$ and hence $q\widetilde{m}\not\in A(S)$. It follows that $A(S)$ has no small elements.

Finally consider the case when there are elements in $S\cap [0,B(S)]$ that are not multiples of $\widetilde{m}$. Let $y$ be the largest such element, say $y=k\widetilde{m}+r'$ with $1\leq r'\leq \widetilde{m}-1$. Consider some $x=q\widetilde{m}\in S$ with $1\leq q\leq n$.

Now if $q\leq n-k$ then $$x+y
\leq (n-k)\widetilde{m}+k\widetilde{m}+\widetilde{m}-1=B(S)-1.$$
Moreover $x+y\equiv r'\not\equiv 0(mod\;\widetilde{m})$, therefore the maximality of $y$ implies that $x+y\not\in S$ and hence $x\not\in A(S)$.

Otherwise we have $q>n-k$. We know that $y\in S$ implies
$$(n+1-k)\widetilde{m}-r'=B(S)-y\in \widetilde{S}$$
and since $\widetilde{S}$ is a numerical semigroup we see that $(q+1)\widetilde{m}-r'\in \widetilde{S}$. Also $(q+1)\widetilde{m}-r'< (n+1)\widetilde{m}=B(S)$, this in turn tells us that
$$(n-q)\widetilde{m}+r'=B(S)-((q+1)\widetilde{m}-r')\in S.$$ Next we have $x+(n-q)\widetilde{m}+r'=n\widetilde{m}+r'$. This quantity is strictly between $b$ and $B(S)$ and hence it is not in $S$. Therefore $x\notin A(S)$ and hence $A(S)\cap [1,B(S)-1]=\emptyset$.

We have shown that $S$ has at most one small atom and if $S$ is not a numerical semigroup then $A(S)\cap [1,B(S)-1]=\emptyset$. Assuming that $S$ is not a numerical semigroup we see that $A(S)$ is either $\{0,F(S)+1\rightarrow\}$ or $\{0,B(S),F(S)+1\rightarrow\}$, in particular $A(S)$ has at most one small atom.
\end{proof}

It is easy to see that $B(S)\in A(S)$ if and only if $S\cap [1,F(S)-B(S)]\neq\emptyset$. Therefore given that $\widetilde{S}$ is a numerical semigroup and $S$ is not, we see that if $[1,F(S)-B(S)]\cap S\neq\emptyset$ then $A(S)=\{0,F(S)+1\rightarrow\}$ and if $[1,F(S)-B(S)]\cap S=\emptyset$ then $A(S)=\{0,B(S),F(S)+1\rightarrow\}$.

\begin{customcor}{4.2}\label{both NS}
Given a numerical semigroup $S$, its complement $\widetilde{S}$ is a numerical semigroup if and only if $S$ has at most one small atom.
\end{customcor}
\begin{proof}
We have already seen one direction that if $S$ is a numerical semigroup with at most one small atom then $\widetilde{S}$ is a numerical semigroup. Now for the other direction assume that $S$ is a numerical semigroup for which $\widetilde{S}$ is also a numerical semigroup. Then by Theorem \ref{Stilde NS} we see that $A(S)$ has at most one small atom, but since $A(S)=S$ we are done.
\end{proof}

\section{Sequence of Complements}

Given a numerical set $S$, we can construct of numerical sets by repeatedly applying the complement operation. The sequence is $S$, $\widetilde{S}$, $\widetilde{\widetilde{S}},\dots$. However since $g(\widetilde{S})<g(S)$, we know that the sequence would eventually reach $\mathbb{N}$. Denote $S^{(0)}=S$ and $S^{(i+1)}=\widetilde{S^{(i)}}$ provided $S^{(i)}\neq \mathbb{N}$. We will show that the length of this sequence is the number of boxes in the Young diagram of $S$ that have hook number $1$. Denote the number of boxes with hook number $1$ as $c_1(S)$.

\begin{proposition}
Given a numerical set $S\neq \mathbb{N}$, $c_1(\widetilde{S})=c_1(S)-1$.
\end{proposition}
\begin{proof}
In the Young diagram of S, boxes with hook number 1 appear at the corners (see Figure \ref{fig:Hook no 1}). Moreover when one travels along the curve between the Young diagrams of $S$ and $\widetilde{S}$ and looks at the boxes with hook number $1$, they alternately belong to $S$ and $\widetilde{S}$. Since the first and last one both belong to $S$, it follows that $c_1(S)=c_1(\widetilde{S})+1$.
\end{proof}
\begin{figure}[h]
    \centering
    \includegraphics{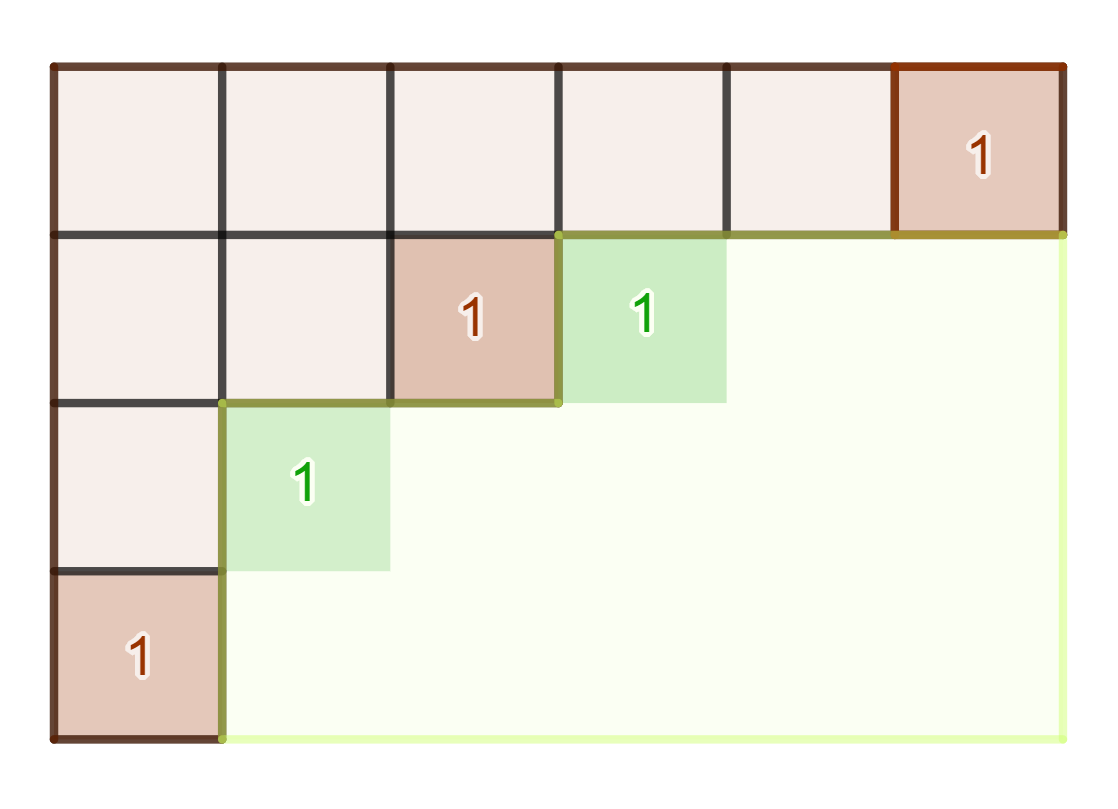}
    \caption{The Young diagrams of a numerical set and its complement with boxes of hook number $1$ marked.}
    \label{fig:Hook no 1}
\end{figure}

\begin{customcor}{5.2}
For any numerical set $S$, $S^{(c_1(S))}=\mathbb{N}$.
\end{customcor}
\begin{proof}
For $n<c_1(S)$, $c_1(S^{(n)})=c_1(S)-n>0$ so $S^{(n)}\neq\mathbb{N}$. Moreover $c_1(S^{(c_1(S))})=0$, so $S^{(c_1(S))}=\mathbb{N}$.
\end{proof}

\begin{customprop}{5.3}
Let $S$ be a numerical set, then $A(S)\subseteq A(S^{(2)})$.
\end{customprop}
\begin{proof}
The Young diagram of $S^{(2)}$ is obtained from the Young diagram of $S$ be deleting several top rows and left columns. Therefore the set of hook numbers of $S^{(2)}$ is a subset of hook numbers of $S$. As mentioned earlier \cite{Kaplan et al} proved that the hook numbers of the Young diagram of $S$ are precisely the gaps of $A(S)$. It follows that $A(S)\subseteq A(S^{(2)})$.
\end{proof}
\begin{customlem}{5.4}\label{max ED}
A numerical semigroup $S$ is of max embedding dimension if and only if $S'=\{x-m(S)\mid x\in S, x>0\}$ is a numerical semigroup.
\end{customlem}
\begin{proof}
See \cite{NS Book}. 
\end{proof}

\begin{customprop}{5.5}
Given a numerical semigroup $S$, $S^{(2)}$ is a numerical semigroup if and only if $S\cup [B(S),\infty)$ is of max embedding dimension.
\end{customprop}
\begin{proof}
Since $S$ is a numerical semigroup we know that $1\not\in S$ and hence $B(\widetilde{S})=B(S)-m(S)$. It follows that
$$S^{(2)}=\{x-m(S)\mid x\in S\cup [B(S),\infty), x>0\}.$$
Now Lemma \ref{max ED} tells us that $S^{(2)}$ is a numerical semigroup if and only if $S\cup [B(S),\infty)$ has max embedding dimension.
\end{proof}

\begin{customcor}{5.6}
If $S$ is a numerical semigroup of max embedding dimension then $S^{(2)}$ is also a numerical semigroup.
\end{customcor}
\begin{proof}
If $S$ has max embedding dimension then Lemma \ref{max ED} tells us that $S'=\{x-m(S)\mid x\in S, x>0\}$ is a numerical semigroup and hence $S^{(2)}=S'\cup [B(S)-m(S),\infty)$ is also a numerical semigroup.
\end{proof}

\begin{customprop}{5.7}
Given a numerical semigroup $S$ and a positive integer $n$ there exists another numerical semigroup $T$ for which $T^{(2n)}=S$.
\end{customprop}
\begin{proof}
We prove this for $n=1$, the general case follows by induction. Let $T=\{0\}\cup\{x+m(S)\mid x\in S\}\setminus \{F(S)+2m(S)\}$, then $m(T)=m(S)$, $F(T)=F(S)+2m(S)$ and $B(T)=F(S)+2m(S)-1$. Next if $x,y\in S$ then $x+y+m(S)\in S$ and $x+y\neq F(S)$ which means that $x+m(S)+y+m(S)\in T$. It follows that $T$ is closed under addition and hence is a numerical semigroup. Finally
$$T\cup [B(T),\infty) =\{0\}\cup\{x+m(S)\mid x\in S\},$$
and hence $T^{(2)}=S$.
\end{proof}

Several other questions can be explored about the sequence of complements. For example,\\
1) If $S$ is a numerical set and $S^{(2)}$ is a numerical semigroup then what can be said about $A(S)$?\\
2) More generally for a fixed $n\geq 2$, if $S$ is a numerical set and $S^{(n)}$ is a numerical semigroup then what can be said about $A(S)$?\\
3) For a fixed odd $n\geq 3$, classify numerical semigroups $S$ for which there is another numerical semigroup $T$ with $T^{(n)}=S$.

\section{Acknowledgements}
This paper is based on research done in the San Diego State University REU 2019. It was supported by NSF-REU award 1851542.

\end{document}